\documentclass[10pt]{article}
\usepackage[english]{babel}
\usepackage[T1]{fontenc}
\usepackage[latin1]{inputenc}
\usepackage{latexsym}
\usepackage{stmaryrd}
\usepackage{amsmath,amsfonts,amssymb,amsthm,psfrag}
\usepackage{paralist} %liste numerotee
\usepackage[dvips]{graphicx,epsfig}
\usepackage{grffile}
\usepackage{color}

\newtheorem{defi}{Definition}

\newtheorem{thm}{Theorem}

\newtheorem{prop}{Proposition}

\newtheorem{lem}{Lemma}

\newtheorem{rem}{Remark}
\newcommand{\conv}[1]{\textrm{conv}}
\newcommand{\var}[1]{\textrm{Var}}

\newcommand{\E}{{\mathbb E}}
\renewcommand{\P}{{\mathbb P}}

\title{Random symmetrizations of convex bodies}

\author{D. Coupier, Yu. Davydov\\
{\small  Universit\'e Lille 1, Laboratoire Paul
Painlev\'e}
}
\begin{document}
\maketitle

\begin{abstract}
In this paper, the asymptotic behavior of sequences of successive Steiner and Minkowski symmetrizations is investigated. We state an equivalence result between the convergences of those sequences for Minkowski and Steiner. Moreover, in the case of independent (and not necessarily identically distributed) directions, we prove the almost sure convergence of successive symmetrizations at rate exponential for Minkowski, and at rate $e^{-c\sqrt{n}}$, with $c>0$, for Steiner.

\end{abstract}


%------------------------------------------------------------------------------
\section{Introduction}
%------------------------------------------------------------------------------

Let $A$ be a convex body of $\mathbb{R}^d$, i.e. a convex compact set with nonempty interior, and $u\in\mathbb{S}^{d-1}$ be a unit vector. The set $A$ can be considered as a family of line segments parallel to the direction $u$. Sliding these segments along $u$ and centering them with respect to the hyperplan $u^{\perp}$, gives $S_{u}A$, the \textit{Steiner symmetral} of $A$.

\textit{Steiner symmetrization} play an important role in geometry and its applications. Indeed, this transformation possesses certain contraction properties which allow in many cases to round off the initial set after multiple applications. Moreover, the limiting ball delivers the solution of several optimization problems, as for instance the Isoperimetric Inequality, the Brunn-Minkowski Inequality and the Blaschke-Santal\'o Inequality (see Section 9.2 of Gruber \cite{G}).

Another important transformation is the \textit{Minkowski} (sometimes called Bla\-schke) \textit{symmetrization}. The \textit{Minkowski symmetral} of a convex body $A$ with direction $u\in\mathbb{S}^{d-1}$, denoted by $B_{u}A$, is defined as the arithmetic mean of $A$ and $\pi_{u}(A)$, its orthogonal symmetric with respect to $u^{\perp}$.\\

Our aim is to study the asymptotic behavior of successive Steiner and Minkowski symmetrizations. Recently this question has received considerable development. Without applying for completeness, we will note here a few works characterizing the main tendencies.

Among the works concerning deterministic sequences of directions, let us mention Klain \cite{Klain}. When the directions are chosen among a finite set, he stated the convergence of the sequence of successive Steiner symmetrals to a limiting set which is symmetric under reflection in any of the directions that appear infinitely often in the sequence. In \cite{BKLYZ}, Bianchi et al proved that, from any dense set of directions (in $\mathbb{S}^{d-1}$), it is always possible to extract a countable sequence rounding off any convex body by successive Steiner symmetrizations. They also exhibited countable dense sequences of directions and convex bodies whose corresponding sequences of Steiner symmetrals do not converge at all (the order of directions matters !).

The case of random Steiner symmetrizations has also been investigated. The first result (to our knowledge) is due to Mani Levitska \cite{ML} and concerns the case of i.i.d. directions, chosen uniformly on the sphere $\mathbb{S}^{d-1}$. He establised the a.s. convergence of the sequence of successive Steiner symmetrals of any convex body to a ball. In \cite{V}, Vol\v{c}i\v{c} has extended Mani Levitska's result to measurable sets with finite measure, and to any probability measure assigning positive mass to any open set of $\mathbb{S}^{d-1}$. Let us cite the Burchad and Fortier's paper \cite{BF} in which they stated that the a.s. convergence still occurs for (independent but) non identically distributed directions provided they satisfy some restrictive condition (see (\ref{nonstat}) further). Combining a probabilistic approach and the powerful analytical device of spherical harmonics, Klartag stated in his remarkable article \cite{K}, a rate of convergence for successive Steiner symmetrizations. Precisely, for any given convex body $A$, there exists an (implicit) sequence of $n$ directions such that the Hausdorff distance between the resulting sequence of successive Steiner symmetrals and the limiting ball is smaller than $e^{-c\sqrt{n}}$, with $c>0$. As a key step, Klartag proved a similar result for successive Minkowski symmetrizations, but at exponential rate.\\

Our first result (Theorem \ref{thm:Steiner}) complements and strengthens the results of \cite{K,ML,V}. Indeed, it affirms that the convergence of the sequence of successive Steiner symmetrizations is almost sure on the one hand, and at rate $e^{-c\sqrt{n}}$ on the other hand. Moreover, the random directions are allowed to be non identically distributed and their distributions may avoid some open sets of the sphere $\mathbb{S}^{d-1}$, which is forbidden in \cite{BF,V}. The independence hypothesis of directions can also be relaxed (see Remark \ref{Rmk:indephypo}). The proof of Theorem \ref{thm:Steiner} substantially follows the ideas of Klartag \cite{K}. Firstly, we state the a.s. convergence of successive Minkowski symmetrizations at exponential rate (Theorem \ref{thm:Minkowski}). The main advantage of Minkowski symmetrization over Steiner is to exhibit a (strict) contraction property (see Proposition \ref{prop:contraction}) from which Theorem \ref{thm:Minkowski} straight derives. Thus, the passage from Minkowski to Steiner is only based on the inclusion $S_{u}A\subset B_{u}A$. This explains the loss in rate of convergence between Minkowski and Steiner.

Our second result (Theorem \ref{thm:equiv}) provides a surprising link between Steiner and Minkowski symmetrizations. A sequence of directions $(u_n)_{n\in\mathbb{N}}$ is said to be \textit{$S-$universal} if, for any $k$, the sequence of successive Steiner symmetrizations corresponding to the shifted sequence  $(u_{k+n})_{n\in\mathbb{N}}$ rounds off any convex body. In the same way, the concept of \textit{$M-$universal} sequence (for Minkowski) is introduced. Theorem \ref{thm:equiv} says that the concepts of $S$ and $M-$universality coincide; we will then omit the prefixes $S$ and $M$. This allows in many cases to deduce from known results about the Steiner symmetrization, new results about the Minkowski symmetrization. For example, from aforementioned Mani Levitska's result \cite{ML} about random i.i.d. Steiner symmetrizations, we immediately receive a similar result for Minkowski symmetrizations, without the sophisticated use of spherical harmonics (Proposition \ref{prop:equiv1}). Theorem \ref{thm:equiv} also allows to transfer results of \cite{BKLYZ,BF} to Minkowski symmetrizations. In particular, any dense set of directions (in $\mathbb{S}^{d-1}$) contains a universal subsequence (Proposition \ref{prop:dense}).\\

The paper is organized as follows. Section \ref{sect:Definitions} contains precise definitions of Steiner and Minkowski symmetrizations and their preliminary properties. In Section \ref{sect:Equivalence}, the concepts of $S$ and $M-$universal sequences are introduced. Theorem \ref{thm:equiv} is proved and applied in two different contexts; random (Propositions \ref{prop:equiv1} and \ref{prop:BCuniversal}) and deterministic (Propositions \ref{prop:dense} and \ref{prop:nonCV}). Sections \ref{sect:RandomMinkowski} and \ref{sect:RandomSteiner} are devoted to random symmetrizations (respectively Minkowski and Steiner symmetrizations). The proof of Proposition \ref{prop:contraction}, rather long and thechnical, is addressed in Section \ref{sect:PropContraction}. Finally, some open questions are formulated in Section \ref{sect:Open}.

\section{Steiner and Minkowski symmetrizations}
\label{sect:Definitions}

This section contains the definitions of Steiner and Minkowski symmetrizations and their basic properties. Let us denote by $\mathcal{K}_d$ the set of convex bodies of $\mathbb{R}^{d}$.

\begin{defi}
\label{defi:Steiner}
Let $A\in\mathcal{K}_{d}$ and $u\in\mathbb{S}^{d-1}$. The convex body $A$ can be considered as a family of line segments parallel to the direction $u$. Sliding each of these segments along $u$ so that they become symmetrically balanced around the hyperplane $u^\perp$, a new set is obtained, called the \textbf{Steiner symmetral} of $A$ with direction $u$ and denoted by $S_{u}A$. The mapping $S_u$ defined on $\mathcal{K}_{d}$ is called \textbf{Steiner symmetrization} with direction $u$.
\end{defi}

It derives from Definition \ref{defi:Steiner} that Steiner symmetrization preserves the volume: for any  $A\in\mathcal{K}_{d}$ and $u\in\mathbb{S}^{d-1}$,
\begin{equation}
\label{VolumeInvariance}
vol(S_{u}A) = vol(A)
\end{equation}
(where $vol(A)$ denotes the $d-$dimensional Lebesgue measure $\lambda^d$ of the measurable set $A$).

\begin{figure}[!ht]
\begin{center}
\psfrag{a}{\small{$A$}}
\psfrag{b}{\small{$S_{u}A$}}
\psfrag{c}{\small{$u^{\perp}$}}
\includegraphics[width=5cm,height=6.5cm]{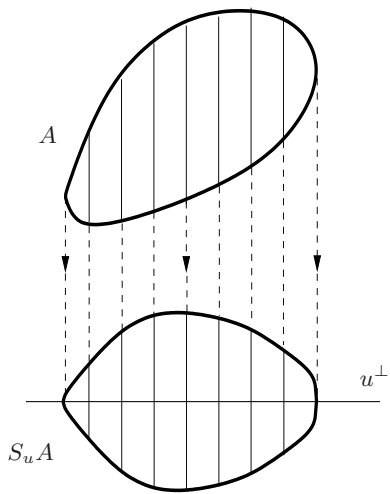}
\end{center}
\caption{\label{fig:Steiner} {\small \textit{Steiner symmetrization with direction $u$. The dotted lines represent the sliding of orthogonal segments along $u$.}}}
\end{figure}

Let us denote by $\pi_u$ the orthogonal reflection operator with respect to the hyperplane $u^{\perp}$:
$$
\forall x \in \mathbb{R}^{d} , \pi_u(x) = x - 2 \langle x , u \rangle ~,
$$
where $\langle \cdot, \cdot\rangle$ is the scalar product in $\mathbb{R}^d$.

\begin{defi} 
\label{defi:Minkowski}
Let $A\in\mathcal{K}_{d}$ and $u\in\mathbb{S}^{d-1}$. The \textbf{Minkowski symmetral} of $A$ with direction $u$, denoted by $B_{u}A$, is defined by
\begin{equation}
\label{Mink}
B_uA = \frac{1}{2} ( A \oplus \pi_u(A) ) ~,
\end{equation}
where $\oplus$ denotes the Minkowski sum of sets $A$ and $B$. The mapping $B_u$ defined on $\mathcal{K}_{d}$ is called \textbf{Minkowski symmetrization} with direction $u$.
\end{defi}

The \textit{support function} $f_{A}$ of a convex body $A\in\mathcal{K}^{d}$ is defined by
$$
f_A(\theta)= \sup_{x\in A}\langle x, \theta\rangle , \; \mbox{ for any } \, \theta \in S^{d-1} ~.
$$
The support functions are a useful tool in Convex geometry. In particular, any convex body is characterized by its support function (see Theorem 4.3 p.57 of \cite{G}). Let $\sigma$ be the Haar probability measure on $\mathbb{S}^{d-1}$. The value
$$
L_A = \int_{\mathbb{S}^{d-1}} f_A d\sigma
$$
is called the \textit{mean radius} of $A$.

Minkowski symmetrization presents an advantage over Steiner symmetrization. Classical properties of support functions (see Proposition 6.2 p.81 of \cite{G}) allows to express $f_{B_{u}A}$ as the arithmetic mean of $f_{A}$ and $f_{\pi_{u}A}$:
\begin{equation}
\label{SupportFunctionMink}
f_{B_{u}A} = \frac{1}{2} ( f_{A} + f_{\pi_u(A)} ) ~.
\end{equation}
As a consequence of (\ref{SupportFunctionMink}), Minkowski symmetrization preserves the mean radius: for any $A\in\mathcal{K}_{d}$ and $u\in\mathbb{S}^{d-1}$,
\begin{equation}
\label{MeanRadiusInvariance}
L(S_{u}A) = L(A) ~.
\end{equation}

Let $B(x,r)$ be the euclidean closed ball with center $x$ and radius $r$. Let $D=B(0,1)$ the unit ball and $v_d$ its volume. The reader may refer to \cite{G} for details about the following properties.

\begin{lem}
\label{lem:BasicProp}
Let $A\in\mathcal{K}_{d}$ and $u\in\mathbb{S}^{d-1}$.
\begin{itemize}
\item[$(i)$] $S_{u}A$ and $B_{u}A$ are convex bodies, symmetric with respect to $u^{\perp}$.
\item[$(ii)$] Let $A'\in\mathcal{K}_{d}$ containing $A$. Then, $S_{u}A\subset S_{u}A'$ and $B_{u}A\subset B_{u}A'$. In particular, if $R(A)$ denotes the \textit{circumradius} of $A$, i.e.
$$
R(A)=\inf\{R>0 , \, A \subset B(0,R)\} ~,
$$
then $S_{u}A$ and $B_{u}A$ are included in the centered ball $R(A)D$.
\item[$(iii)$] $S_{u}A$ is included in $B_{u}A$.
\end{itemize}
\end{lem}
 
The inclusion $S_{u}A\subset B_{u}A$ can be understood as follows. Let $\Delta$ be one of the orthogonal segments to $u^\perp$ which compose $S_{u}A.$ It is obtained by translation of a segment $\Delta'$ composing $A$ (see Definition \ref{defi:Steiner}). Then,
$$
\Delta = \frac{1}{2}(\Delta' \oplus \pi_u(\Delta')) \, \subset \, \frac{1}{2}(A + {\pi_u(A)}) = B_uA
$$
and Lemma \ref{lem:BasicProp} $(iii)$ follows.

We deduce immediately from identities (\ref{VolumeInvariance}) and (\ref{MeanRadiusInvariance}), and Lemma \ref{lem:BasicProp} $(iii)$ that Steiner symmetrization decreases the mean radius whereas Minkowski symmetrization increases the volume:
\begin{equation}
\label{MeanRadVol}
L(S_uA) \leq L(A) \; \mbox{ and } \; vol(B_uA) \geq vol(A) ~.
\end{equation}

Two classical metrics on the set $\mathcal{K}_d$ are involved in our proofs; the Hausdorff distance
$$
d_H(A,B) = \max \left\{ \inf \{ \varepsilon>0 \;|\; A \subset B\oplus B(0,\varepsilon) \} \;,\; \inf \{ \varepsilon>0 \;|\; B \subset A\oplus B(0,\varepsilon) \} \right\}
$$
and the Nikod\'ym distance
$$
d_N(A,B) = \lambda^d (A \Delta B) = \lambda^{d}(A\setminus B) + \lambda^{d}(B\setminus A) ~.
$$
These distances generate on $\mathcal{K}_d$ the same topology. Hence, all the convergences stated in the sequel correspond to this topology. Some inequalities about Hausdorff and Nikod\'ym distances used in our proofs are addressed in Appendix \ref{appendix}.

\section{Theorem of equivalence}
\label{sect:Equivalence}

Let $(u_n)_{n\geq 1}$ be a sequence of elements of $\mathbb{S}^{d-1}$. For integers $n\geq k\geq1$, we denote by $S_{k,n}$ the sequence of $n-k+1$ consecutive Steiner symmetrizations from $u_{k}$ to $u_{n}$:
$$
S_{k,n}A = S_{u_{n}}(\ldots S_{u_{k+1}}(S_{u_{k}} A)\ldots)
$$
where $A$ is a convex body. When $k=1$, $S_{1,n}A$ is merely denoted by $S_{n}A$. For Minkowski symmetrizations, notations $B_{k,n}A$ and $B_{n}A$ are defined as above.

Let $r(A)$ be the real number such that the ball $r(A)D$ has the same volume as $A$. Recall the set $\mathcal{K}_{d}$ of convex bodies is endowed with the Hausdorff distance. A sequence $(u_n)_ {n\geq 1}$ \textbf{$S-$rounds} the set $A\in\mathcal{K}_{d}$ if
$$
S_nA \rightarrow r(A)D
$$
and \textbf{$M-$rounds} $A$ if
$$
B_nA \rightarrow L(A)D
$$
as $n$ tends to infinity. A sequence $(u_n)_{n\geq 1}$ \textbf{strongly $S-$rounds} the convex body $A$ if, for any $k$,
$$
S_{k,n} A \rightarrow r(A)D
$$
as $n$ tends to infinity. The same terminology holds for Minkowski symmetrizations; $(u_n)_{n\geq 1}$ \textbf{strongly $M-$rounds} $A$ if, for any $k$,
$$
B_{k,n} A \rightarrow L(A)D
$$
as $n$ tends to infinity. Finally, $(u_n)_{n\geq 1}$ is said \textbf{$S-$universal} (respectively \textbf{$M-$universal}) if it strongly $S$-rounds (respectively strongly $M$-rounds) any $A$ of $\mathcal{K}_{d}$.

The next result shows that the notions of $S$ and $M-$universality coincide. Then, such a sequence will be merely said \textbf{universal}.

\begin{thm}
\label{thm:equiv}
A sequence $(u_n)_{n\geq 1}$ of $\mathbb{S}^{d-1}$ is $S-$universal if and only if it is $M-$universal.
\end{thm}

\begin{proof}
We only focus on the sufficient condition because the necessary one is similar.\\
Let $A$ be a convex body. Since Minkowski symmetrization increases the volume, the sequence $(vol(B_{n}A))_{n\geq 1}$ is nondecreasing. Let $V$ its limit. The sets $B_{n}A$, for $n\geq 1$, are all included in the compact set $\{K\in\mathcal{K}_{d} , K\subset R(A)D\}$ of $(\mathcal{K}_{d},d_{H})$ (see Theorem 1.8.4 p.49 in \cite{S} for details). So, $(B_{n}A)_{n\geq 1}$ admits a convergent subsequence $(B_{n_{k}}A)_{k\geq 1}$. Let $E$ its limit. Since the volume is a continuous function on $(\mathcal{K}_{d},d_{H})$, the volume of $E$ equals $V$.\\
For any $m>k$, Lemma \ref{lem:BasicProp} $(iii)$ implies
$$
S_{n_k+1,n_{m}}(B_{n_{k}}A) \subset B_{n_k+1,n_{m}}(B_{n_{k}}A) =
B_{n_{m}}A ~.
$$
By $S-$universality, the left-hand side of the above inclusion converges to the ball $r(B_{n_{k}}A)D$ whereas the right one converges to $E$. Hence, the set $E$ contains $r(B_{n_{k}}A)D$ whose volume tends to $V$ as $k$ tends to infinity. This forces $E$ to be the ball of volume $V$.\\
As a result, any convergent subsequence of $(B_{n}A)_{n\geq 1}$ has the same limit $r(V)D$. By compactness, this also holds for the sequence $(B_{n}A)_{n\geq 1}$ itself. Thus, we indentify $r(V)$ to $L(A)$ by Lemma \ref{lem:HausdorffSupportFct}: as $n\to\infty$,
$$
L(A) =  \int_{\mathbb{S}^{d-1}} f_{B_{n}A} d\sigma \; \to \; \int_{\mathbb{S}^{d-1}} f_{r(V)D} d\sigma = r(V) ~.
$$
Finally, for any $k$, applying the same strategy to the $S-$universal sequence $(u_{k+n})_{n\geq 1}$, we get that $B_{k+1,n} A$ tends to $L(A)D$. The $M-$universal character of $(u_{n})_{n\geq 1}$ follows.\end{proof}

\vskip0.2cm

In what follows, Theorem \ref{thm:equiv} is applied in two different contexts; random (Propositions \ref{prop:equiv1} and \ref{prop:BCuniversal}) and deterministic (Propositions \ref{prop:dense} and \ref{prop:nonCV}). For the first three results, a sufficient condition for the sequence of directions is given, ensuring its universal character. The fourth result concerns the dimension $2$: there exists a uniformly distributed sequence on $\mathbb{S}^{1}$ which is not universal.

\begin{prop}
\label{prop:equiv1}
Let $(U_{n})_{n\geq 1}$ be a stationary sequence of random variables of $\mathbb{S}^{d-1}$, i.e. for any $k$, the sequences $(U_{n})_{n\geq 1}$ and $(U_{k+n})_{n\geq 1}$ are identically distributed. Assume that, for any convex body $A$, $(U_{n})_{n\geq 1}$ a.s. $S-$rounds $A$. Then, $(U_{n})_{n\geq 1}$ is a.s. universal. The same conclusion holds when the $S-$rounding hypothesis is replaced with the $M-$rounding one.
\end{prop}

\begin{proof}
We only check the result under the $S-$rounding hypothesis, the proof under the $M-$rounding hypothesis being very similar.\\
Let $(C_{j})_{j\geq 1}$ be a countable dense subset of the separable set $(\mathcal{K}_{d},d_{H})$. By hypothesis, for any index $j$ and any positive rational number $\varepsilon$, there exists an event of probability $1$ on which $(U_{n})_{n\geq 1}$ $S-$rounds $C_{j}^{\varepsilon}:=C_{j}\oplus B(0,\varepsilon)$. Let $\Omega_{0}$ be the intersection of these events. We are going to prove that on $\Omega_{0}$, $(U_{n})_{n\geq 1}$ $S-$rounds any convex body.\\
Let $A\in\mathcal{K}_{d}$. By compactness, let us consider a convergent subsequence $(S_{n_{k}}A)_{k\geq 1}$ of $(S_{n}A)_{n\geq 1}$ whose limit is denoted by $E$. Since the volume is a continuous function on $(\mathcal{K}_{d},d_{H})$,
\begin{equation}
\label{EgalVolume}
vol(E) = \lim_{k\to\infty} vol(S_{n_{k}}A) = vol(A) ~.
\end{equation}
Let $\varepsilon>0$ be a rational number. There exists an index $j=j(\varepsilon)$ such that $A$ is included in $C_{j}^{\varepsilon}$. Hence, for any $k$,
$$
S_{n_k} A \subset S_{n_k}(C_{j}^{\varepsilon}) ~.
$$
When $k$ tends to infinity and on $\Omega_{0}$, the above inclusion becomes $E\subset r(C_{j}^{\varepsilon})D$. Taking $\varepsilon\searrow 0$, it follows $E\subset r(A)D$. By (\ref{EgalVolume}), this is possible only if $E=r(A)D$. We conclude by compactness that $(U_{n})_{n\geq 1}$ $S-$rounds any $A\in\mathcal{K}_{d}$ on the event $\Omega_{0}$ of probability $1$.\\
By stationarity, for any $k$, this proof applies to $(U_{k+n})_{n\geq 1}$: there exists an event $\Omega_{k}$ of probability $1$ on which the sequence $(U_{k+n})_{n\geq 1}$ $S-$rounds any $A\in\mathcal{K}_{d}$. Hence, by Theorem \ref{thm:equiv}, $(U_{n})_{n\geq 1}$ is universal on $\cap_k\Omega_k$.
\end{proof}

When the random variables $U_{n}$, $n\geq 1$, are independent the hypothesis of stationarity on the sequence $(U_{n})_{n\geq 1}$ can be weakened. The following condition has been introduced by A. Bouchard and M. Fortier \cite{BF}: for any $r>0$ and any sequence $(v_n)_{n\geq 1}$ in $\mathbb{S}^{d-1}$,
\begin{equation}
\label{nonstat}
\sum_{n=1}^{\infty} \P \left( U_{n} \in B(v_{n},r) \right) = \infty ~. 
\end{equation} 
Thanks to the Borel-Cantelli lemma, condition (\ref{nonstat}) implies each open ball $V$ of the sphere $\mathbb{S}^{d-1}$ with positive radius is a.s. infinitely often visited by the $U_{n}$'s. Bouchard and Fortier stated (Corollary 1 of \cite{BF}) that a sequence $(U_{n})_{n\geq 1}$ of independent random variables satisfying (\ref{nonstat}) a.s. $S-$rounds any convex body $A$. Theorem \ref{thm:equiv} extends their result to Minkowski symmetrizations.

\begin{prop}
\label{prop:BCuniversal}
Let $(U_{n})_{n\geq 1}$ be a sequence of independent random variables of $\mathbb{S}^{d-1}$ satisfying condition (\ref{nonstat}). Then, $(U_{n})_{n\geq 1}$ is a.s. universal.
\end{prop}

In the case of i.i.d. directions, Theorems \ref{thm:Minkowski} and \ref{thm:Steiner} specify the rate of convergence, but the price to pay is more high.\\

In \cite{BKLYZ}, G. Bianchi et al proved that each countable dense subset $T\subset\mathbb{S}^{d-1}$ of directions contains a (deterministic) sequence $(u_{n})_{n\geq 1}$ $S-$rounding any given convex body $A$. This result is strengthened here and, using Theorem \ref{thm:equiv}, it is extended to Minkowski symmetrizations.

\begin{prop}
\label{prop:dense}
Every countable dense subset $T\subset\mathbb{S}^{d-1}$ contains a universal sequence.
\end{prop}

\begin{proof}
Let $R>0$. The set $\mathcal{K}_{d}(R)$ of convex bodies having the same volume as the unit ball $D$ and whose circumradius is smaller than $R$ is compact in $(\mathcal{K}_{d},d_{H})$. Given $\varepsilon>0$, we consider a finite $\varepsilon$-net of $\mathcal{K}_{d}(R)$, say $A_{1},\ldots,A_{m}$ with $m=m(\varepsilon,R)$. The result of \cite{BKLYZ} applied to $A_1$ ensures the existence of directions $u_1,\ldots,u_{n_1}$ of $T$ such that
\begin{equation}
\label{incl1}
(1-\varepsilon)D \subset S_{n_1}A_1\subset (1+\varepsilon)D ~.
\end{equation}
Applied to $S_{n_1}A_2$, it provides directions $u_{n_1+1},\ldots,u_{n_2}$ of $T$ such that 
$$
(1-\varepsilon)D \subset S_{n_1+1, n_2}(S_{n_1}A_2)= S_{n_2}A_2 \subset (1+\varepsilon)D ~.
$$
Steiner symmetrization increases inradius and decreases circumradius. So statement (\ref{incl1}) becomes:
$$
(1-\varepsilon)D \subset S_{n_2}A_1\subset (1+\varepsilon)D ~.
$$
Hence, we obtain by induction a sequence of $n=n(\varepsilon,R)$ directions $\{u_1,\ldots,u_{n}\}$ of $T$ satisfying for $i=1,\ldots,m$
\begin{equation}
\label{eq1}
(1-\varepsilon)D \subset S_{n}A_i \subset (1+\varepsilon)D ~.
\end{equation}
Let $A$ be a convex body belonging to $\mathcal{K}_{d}(R)$. Let $A_{i_{0}}$ be an element of the $\varepsilon$-net of $\mathcal{K}_{d}(R)$ such that $d_H(A, A_{i_0}) < \varepsilon$. Recall the Nikod\'ym distance $d_{N}$ generates on $\mathcal{K}_{d}$ the same topology as $d_{H}$ (see Appendix \ref{appendix}). Inclusions (\ref{eq1}), Lemmas \ref{lem:Nykodim1Lip} and \ref{lem:Nyk<Hausdorff} imply
\begin{eqnarray}
\label{StrategyCepsilon}
d_N(S_{n}A, D) & \leq & d_N(S_{n}A, S_{n}A_{i_0}) + d_N(S_{n}A_{i_0} , D) \nonumber\\
& \leq & d_N(A, A_{i_0}) + C d_H(S_{n}A_{i_0} , D) \nonumber\\
& \leq & C \varepsilon ~,
\end{eqnarray}
where $C=C(d,R)$ is a positive constant. Now, given a decreasing sequence $(\varepsilon_k)_{k\geq 1}$ tending to $0$, we apply the previous strategy to each term $\varepsilon_k$ in order to get some directions, say $u_{n_k+1},\ldots,u_{n_{k+1}}$ satisfying:
$$
d_N(S_{n_k+1,\,n_{k+1}}A, D) \leq C \varepsilon_k ~.
$$
Note this inequality holds for any $A\in\mathcal{K}_{d}(R)$ and the above constant $C$ is the same as in (\ref{StrategyCepsilon}). Concatenating the blocks $\{u_{n_k+1},\ldots,u_{n_{k+1}}\}$, $k\geq 1$, we build a sequence $(u_{n})_{n\geq 1}$ strongly $S-$rounding any convex bodies with the same volume as $D$:
$$
d_N(S_{l,m} A, D) \leq d_N(S_{n_k+1,\,n_{k+1}}(S_{l,n_k}A), D) \leq C(d,R(A)) \varepsilon_k
$$
whenever $l\geq n_{k}$ and $m\geq n_{k+1}$. Finally, we can affirm that $(u_{n})_{n\geq 1}$ is universal thanks to the identity $S_{u}(rA)=rS_{u}A$ and Theorem \ref{thm:equiv}.
\end{proof}

A sequence $(u_{n})_{n\geq 1}$ of $\mathbb{S}^{1}$ is said \textit{uniformly distributed} on $\mathbb{S}^{1}$ if, for any arc $I$ of the unit disc,
$$
\lim_{m\to\infty} \frac{1}{m} \textrm{Card} \left\{ n\leq m , \; u_{n} \in I \right\} = \sigma(I)
$$
where $\sigma$ denotes the Haar probability measure on $\mathbb{S}^{1}$. In \cite{BBGV}, a uniformly distributed sequence $(u_{n})_{n\geq 1}$ on $\mathbb{S}^{1}$ is exhibited (see Section 5) which does not $S-$round a certain convex body (see Example 2.1 in Section 2). By Theorem \ref{thm:equiv}, this sequence is not universal. In other words,

\begin{prop}
\label{prop:nonCV}
There exist a uniformly distributed sequence $(u_{n})_{n\geq 1}$ on $\mathbb{S}^{1}$ and a convex body $A$ such that $(u_{n})_{n\geq 1}$ does not $M-$round $A$.
\end{prop}

\section{Random Minkowski symmetrizations}
\label{sect:RandomMinkowski}

Let $A$ be a convex compact set in $\mathbb{R}^{d}$. The goal of this section is to state a rate of convergence for
$$
B_{n}A = B_{U_{n}}(\ldots B_{U_{2}}(B_{U_{1}} A)\ldots)
$$
to $L(A)D$ when the random directions $U_{k}\in\mathbb{S}^{d-1}$, $k\geq 1$, are independent.

\subsection{Rate of convergence}
\label{sect:RateMinkowski}

Let $\sigma$ be the Haar probability measure on $\mathbb{S}^{d-1}$.

\begin{thm}
\label{thm:Minkowski}
Assume that, for any $k\geq 1$, the distribution $\nu_{k}$ of $U_k$ is absolutely continuous with respect to $\sigma$ and its density satisfies
\begin{equation}
\label{hypodistrib}
\frac{d\nu_{k}}{d\sigma}(u) \leq \alpha < \frac{d}{d-1}
\end{equation}
for some $\alpha>0$ and $\sigma-$a.e. $u\in\mathbb{S}^{d-1}$. Then, there exists a constant $c>0$ such that, with probability $1$,
\begin{equation}
\label{rateMinkowsi}
\exists n_{0}(\omega) , \; \forall n \geq n_{0} , \; d_{H}\left( B_{n}A , L(A)D \right) \leq e^{-cn} ~.
\end{equation}
Furthermore, the first random integer $n_{0}$ from which the above inequality holds admits exponential moments.
\end{thm}

\begin{rem}
Let us compare our result with Klartag's one. Theorem 1.3 of \cite{K} states that for any $n$, there exists $n$ Minkowski symmetrizations transforming any convex body $A$ into a convex body $A_{n}$ whose distance to $L(A)D$ is smaller than $e^{-\delta n}$ (where $\delta$ is a positive constant). Theorem \ref{thm:Minkowski} offers an advantage with respect to Klartag's result: whereas only one (implicit) sequence of $n$ directions suits in Theorem 1.3 of \cite{K}, almost every realization of $(U_{1},\ldots,U_{n})$ satisfies statement (\ref{rateMinkowsi}).\\
The exponential decrease holds for any $n$ in Theorem 1.3 of \cite{K} and only from a random integer in Theorem \ref{thm:Minkowski}. However, this latter admits exponential moments.
\end{rem}

\begin{rem}
It is worth pointing out here that any real number $c$ such that
$$
0 < c < -\frac{1}{2d} \log \frac{\alpha(d-1)}{d}
$$
satisfies statement (\ref{rateMinkowsi}). See the proof of Theorem \ref{thm:Minkowski} for details.
\end{rem}

\begin{rem}
Finally, let us remark Theorem \ref{thm:Minkowski} still holds when the volume of $A$ is null.
\end{rem}

Let $h_{A}$ be the centered support function of $A$:
\begin{equation}
\label{h_A}
h_{A} = f_{A} - L(A) ~.
\end{equation}
Proposition \ref{prop:contraction} is the heart of the proof of Theorem \ref{thm:Minkowski}. It essentially says that $h_{B_{U}A}$ is a contraction when the random direction $U$ is uniformly distributed on $\mathbb{S}^{d-1}$. The proof of Proposition \ref{prop:contraction} is rather technical and is addressed in Section \ref{sect:PropContraction}.

\begin{prop}
\label{prop:contraction}
Let $U$ be a random variable of $\mathbb{S}^{d-1}$ with distribution $\sigma$. Then,
\begin{equation}
\label{contractionMink}
\E \| h_{B_{U}A} \|_{2}^{2} \leq \frac{d-1}{d} \| h_{A} \|_{2}^{2} ~.
\end{equation}
\end{prop}

Inequality (\ref{contractionMink}) is actually an equality when $d=2$ and $d\to\infty$. The case $d=2$ is treated at the beginning of Section \ref{sect:PropContraction}. In higher dimension, a vector $U$ chosen uniformly on $\mathbb{S}^{d-1}$ is (almost) orthogonal to a given $v$ with large probability:
$$
\pi_{U}(v) = v - 2 \langle v , U \rangle U
$$
is close to $v$ with a probability tending to $1$. We can then check that $\E\|h_{B_{U}A}\|_{2}^{2}$ is larger than $\|h_{A}\|_{2}^{2}-o(1)$ as $d\to\infty$.\\

Theorem \ref{thm:Minkowski} derives from Proposition \ref{prop:contraction} and the Borel-Cantelli lemma.

\begin{proof}[Proof of Theorem \ref{thm:Minkowski}.]
Let $\rho_{k}$ be the probability density function of $\nu_{k}$ with respect to $\sigma$ and $\alpha_{d}=\frac{\alpha(d-1)}{d}$. Hypothesis (\ref{hypodistrib}) and Proposition \ref{prop:contraction} applied to $B_{U_{1}}A$ thus $A$, imply:
\begin{eqnarray*}
\E \| h_{B_{2}A} \|_{2}^{2} & = & \int_{\mathbb{S}^{d-1}} \left( \int_{\mathbb{S}^{d-1}} \| h_{B_{u_{2}}(B_{u_{1}}A)} \|_{2}^{2} \rho_{2}(u_{2}) d\sigma(u_{2}) \right) \rho_{1}(u_{1}) d\sigma(u_{1}) \\
& \leq & \alpha_{d} \int_{\mathbb{S}^{d-1}} \| h_{B_{u_{1}}A} \|_{2}^{2} \rho_{1}(u_{1}) d\sigma(u_{1}) \\
& \leq & \alpha_{d}^{2} \| h_{A} \|_{2}^{2} ~.
\end{eqnarray*}
By induction, it follows that for any integer $n$,
$$
\E \| h_{B_{n}A} \|_{2}^{2} \leq \alpha_{d}^{n} \| h_{A} \|_{2}^{2} ~.
$$
Lemma \ref{lem:InfiniNorme} below allows to upperbound the expectation of the $L_{\infty}$ norm of $h_{B_{n}A}$. Indeed,
\begin{eqnarray*}
\E \| h_{B_{n}A} \|_{\infty}^{d} & \leq & z_{d} \E \| h_{B_{n}A} \|_{2} \\
& \leq & z_{d} \sqrt{\E \| h_{B_{n}A} \|_{2}^{2}} \\
& \leq & z_{d} \alpha_{d}^{n/2} \| h_{A} \|_{2} ~,
\end{eqnarray*}
where $z_{d}$ denotes the constant in the right-hand side of (\ref{InfiniNorme}). Hence,
\begin{equation}
\label{EspNormeInfini}
\E \| h_{B_{n}A} \|_{\infty} \leq \left( z_{d} \| h_{A} \|_{2} \alpha_{d}^{n/2} \right)^{1/d} ~.
\end{equation}
Markov's inequality and (\ref{EspNormeInfini}) give
\begin{equation}
\label{ProbaNormeInfini}
\P (\| h_{B_{n}A} \|_{\infty} > r^{n}) \leq r^{-n} \E \| h_{B_{n}A} \|_{\infty} \leq \left( z_{d} \| h_{A} \|_{2} \right)^{1/d} \left( r^{-1} \alpha_{d}^{1/2d} \right) ^{n} ~.
\end{equation}
The real number $r>0$ can be chosen such that $\alpha_{d}^{1/2d}<r<1$ by hypothesis (\ref{hypodistrib}). Then, the Borel-Cantelli lemma applies; a.s. for $n$ large enough, $\| h_{B_{n}A} \|_{\infty}$ is smaller than $r^{n}$. Statement (\ref{rateMinkowsi}) follows from the identity
$$
d_{H}\left( B_{n}A , L(A)D \right) = \| h_{B_{n}A} \|_{\infty} ~.
$$
Finally, let us denote by $n_{0}$ the first (random) integer from which the Hausdorff distance between $B_{n}A$ and $L(A)D$ is smaller than $r^{n}$. We deduce from (\ref{ProbaNormeInfini}) that $n_{0}$ admits exponential moments:
$$
\P ( n_{0} > m ) \leq \left( z_{d} \| h_{A} \|_{2} \right)^{1/d} \left( r^{-1} \alpha_{d}^{1/2d} \right) ^{m} ~.
$$
\end{proof}

In order to optimize the rate of convergence with respect to the dimension $d$ in Theorem 1.3 of \cite{K}, Klartag uses technical lemmas to go from $L_{2}$ norm to $L_{\infty}$ norm (see Section 4 of \cite{K}). Here, we only focus our attention on the parameter $n$. So, the following basic result will be suitable.

\begin{lem}
\label{lem:InfiniNorme}
Recall that $R(A)$ denotes the circumradius of $A$. Then, for any integer $n\geq 1$,
\begin{equation}
\label{InfiniNorme}
\| h_{B_{n}A} \|_{\infty}^{d} \leq \frac{d R(A)^{d-1}}{\kappa_{d-1}} \| h_{B_{n}A} \|_{2} ~.
\end{equation}
\end{lem}

\begin{proof}
Classical properties of support functions (namely positive homogeneity of degree $1$ and subadditivity; see the book of Gruber \cite{G} p.57) imply $f_{B_{n}A}$ can be extended to a Lipschitz function defined on the whole space $\mathbb{R}^{d}$. Its Lipschitz constant equals $\| f_{B_{n}A} \|_{\infty}$, i.e. its supremum over $\mathbb{S}^{d-1}$. Since all the $B_{n}A$'s are included in $R(A)D$, the $f_{B_{n}A}$'s are $R(A)-$Lipschitz functions. So do the functions $h_{B_{n}A}$, $n\geq 1$.\\
To conclude, it suffices to remark the $L_{1}$ norm of a $R(A)-$Lipschitz function $f$ defined on $\mathbb{S}^{d-1}$ is larger than the volume of a right cone with height $\| f \|_{\infty}$ over a $(d-1)-$dimensional ball with radius $R(A)^{-1} \| f \|_{\infty}$. In other words,
$$
\| f \|_{2} \geq \| f \|_{1} \geq \frac{\kappa_{d-1} \| f \|_{\infty}^{d}}{d R(A)^{d-1}} ~,
$$
where $\kappa_{d-1}$ denotes the volume of the $(d-1)-$dimensional unit ball. The searched result then follows.
\end{proof}

\subsection{Proof of Proposition \ref{prop:contraction}}
\label{sect:PropContraction}

Recall the support function $f_{B_{u}A}$ can be expressed as the arithmetic mean of $f_{A}$ and $f_{\pi_{u}A}$ (see (\ref{SupportFunctionMink})). Then, using the invariance of the Haar probability measure $\sigma$ under the application $v\mapsto\pi_{u}(v)$, for any $u\in\mathbb{S}^{d-1}$, the $L_{2}$ norm of $h_{B_{u}A}$ satisfies
$$
\| h_{B_{u}A} \|_{2}^{2} = \frac{1}{2} \| h_{A} \|_{2}^{2} + \frac{1}{2} \langle h_{A} , h_{\pi_{u}A} \rangle ~.
$$
Assume $U$ is distributed according to $\sigma$. By Fubini's theorem,
$$
\E \langle h_{A} , h_{\pi_{U}A} \rangle = \int_{\mathbb{S}^{d-1}} h_{A}(v) \left( \int_{\mathbb{S}^{d-1}} h_{A}(\pi_{u} v) d\sigma(u) \right) d\sigma(v) ~.
$$
(indeed $f_{\pi_{u}A}=f_{A}\circ\pi_{u}$). Now, when $d=2$, the probability measure $\sigma$ is also invariant under the application $J_{v}:u\mapsto\pi_{u}(v)$, for any $v\in\mathbb{S}^{1}$. So, the integral
$$
\int_{\mathbb{S}^{1}} h_{A}(\pi_{u} v) d\sigma(u)
$$
is null and so does $\E \langle h_{A} , h_{\pi_{U}A} \rangle$. To sum up, Proposition \ref{prop:contraction} is easily proved in dimension $d=2$ and
$$
\E \| h_{B_{U}A} \|_{2}^{2} = \frac{1}{2} \| h_{A} \|_{2}^{2} ~.
$$

However, this strategy does not hold whenever $d>2$. One can prove in this case that the image measure $\sigma J_{v}^{-1}$ admits a probability density function with respect to $\sigma$ which is unbounded in the vicinity of $v$.\\
Consequently, in order to prove Proposition \ref{prop:contraction}, we follow ideas of Klartag \cite{K} based on spherical harmonics.\\

In the rest of this section, we assume $d>2$. A polynomial $P$ defined on $\mathbb{R}^{d}$ is a \textit{homogeneous harmonic of degree} $k$ if $P$ is a homogeneous polynomial of degree $k$ and is harmonic (i.e. $\Delta P=0$). Let $\mathcal{S}_{k}$ be the following linear space:
$$
\mathcal{S}_{k} = \{ P_{|\mathbb{S}^{d-1}} , \; P \mbox{ is a homogeneous harmonic of degree } k \}
$$
where $P_{|\mathbb{S}^{d-1}}$ denotes the restriction of the polynomial $P$ to the sphere $\mathbb{S}^{d-1}$. The elements of $\mathcal{S}_{k}$ are called \textit{spherical harmonics of degree} $k$. The reader may refer to \cite{M} for complete references about spherical harmonics.

The linear space $L_{2}(\mathbb{S}^{d-1})$ admits the following orthogonal direct sum decomposition:
\begin{equation}
\label{OrthogonalDirectSum}
L_{2}(\mathbb{S}^{d-1}) = \bigoplus_{k\geq 0} \mathcal{S}_{k} ~.
\end{equation}
Let us write the centered support function $h_{A}$ according to (\ref{OrthogonalDirectSum}): $h_{A}=\sum g_{k}$. Thus,
\begin{equation}
\label{ExpansionBuA}
h_{B_{u}A} = \frac{1}{2} \left( h_{A} + h_{A}\circ\pi_{u} \right) = \sum_{k\geq 0} B_{u}g_{k} ~,
\end{equation}
where
$$
B_{u}g_{k} = \frac{1}{2} \left( g_{k} + g_{k}\circ\pi_{u} \right) ~.
$$
First, it is clear that $h_{A}$ is orthogonal to $\mathcal{S}_{0}$. So $g_{0}$ is null. Moreover, from $g_{k}\in\mathcal{S}_{k}$, some elementary computations give $g_{k}\circ\pi_{u}\in\mathcal{S}_{k}$. So does $B_{u}g_{k}$. Hence, (\ref{ExpansionBuA}) is the expansion of $h_{B_{u}A}$ into spherical harmonics, i.e. according to (\ref{OrthogonalDirectSum}). Assume $U$ is distributed according to the Haar probability measure $\sigma$. Then, the searched result follows from Lemma \ref{lem:ContractionSpherical} below and Pythagoras' theorem:
\begin{eqnarray*}
\E \| h_{B_{U}A} \|_{2}^{2} & = & \sum_{k\geq 1} \E \| B_{U}g_{k} \|_{2}^{2} \\
& = & \sum_{k\geq 1} \frac{d-2+k}{d-2+2k} \, \| g_{k} \|_{2}^{2} \\
& \leq & \frac{d-1}{d} \| h_{A} \|_{2}^{2}
\end{eqnarray*}
since $\frac{d-2+k}{d-2+2k}$ is smaller than $\frac{d-1}{d}$ for any $k\geq 1$.

\begin{lem}
\label{lem:ContractionSpherical}
Let $U$ be a random variable distributed according to $\sigma$. Let $k\geq 1$ and $g\in\mathcal{S}_{k}$. Then,
$$
\E \| B_{U}g \|_{2}^{2} = \frac{d-2+k}{d-2+2k} \, \| g \|_{2}^{2}
$$
where $B_{u}g = \frac{1}{2}(g+g\circ\pi_{u})$.
\end{lem}

The above identity is mentioned in \cite{K} but without proof. So the rest of this section is devoted to its proof.\\
For any $v\in\mathbb{R}^{d}$, $\mathcal{S}_{k}^{v}$ is defined as the set of elements $g\in\mathcal{S}_{k}$ symmetric with respect to the hyperplan $v^{\perp}$:
$$
\mathcal{S}_{k}^{v} = \{ g \in \mathcal{S}_{k} , \; g\circ\pi_{v} = g \} ~.
$$
Let us denote by $\mathrm{Proj}_{\mathcal{S}_{k}^{v}}$ the orthogonal projection onto $\mathcal{S}_{k}^{v}$. Then, the orthogonal projection of $g\in\mathcal{S}_{k}$ is actually equal to $B_{v}g$.

\begin{lem}
\label{lem:Projection}
For any $v\in\mathbb{R}^{d}$ and any $g\in\mathcal{S}_{k}$, $B_{v}g = \mathrm{Proj}_{\mathcal{S}_{k}^{v}}(g)$.
\end{lem}

Let us consider two orthonormal bases $(e_{1},\ldots,e_{d})$ and $(v_{1},\ldots,v_{d})$ in $\mathbb{R}^{d}$, and the isometry $\psi$ mapping $e_{i}$ to $v_{i}$, for any $1\leq i\leq d$. Then:

\begin{lem}
\label{lem:Bases}
For any $g\in\mathcal{S}_{k}$, $\mathrm{Proj}_{\mathcal{S}_{k}^{v_{1}}}(g) = \mathrm{Proj}_{\mathcal{S}_{k}^{e_{1}}}(g\circ\psi)\circ\psi^{-1}$.\end{lem}

Let $g\in\mathcal{S}_{k}$. By Lemmas \ref{lem:Projection} and \ref{lem:Bases},
\begin{eqnarray}
\label{ProjectionBases}
\| B_{v_{1}}g \|_{2}^{2} & = & \| \mathrm{Proj}_{\mathcal{S}_{k}^{v_{1}}}(g) \|_{2}^{2} \nonumber\\
& = & \| \mathrm{Proj}_{\mathcal{S}_{k}^{e_{1}}}(g\circ\psi) \|_{2}^{2} \nonumber\\
& = & \sum_{i=1}^{\ell(k)} \left( \int_{\mathbb{S}^{d-1}} g\circ\psi(x) S_{i}(x) d\sigma(x) \right)^{2}
\end{eqnarray}
where $\ell(k)$ and $(S_{1},\ldots,S_{\ell(k)})$ respectively denote the dimension and an orthonormal basis of $\mathcal{S}_{k}^{e_{1}}$.

Besides, assume an orthonormal basis $(v_{1},\ldots,v_{d})$ is chosen uniformly on the orthogonal group $\mathcal{O}(d)$. Then, its first vector $v_{1}$ is distributed uniformly on the sphere $\mathbb{S}^{d-1}$, i.e. according to $\sigma$. Precisely, let $\mu$ be the Haar probability measure on $\mathcal{O}(d)$. Let us denote by $\Psi$ the application from $\mathcal{O}(d)$ to $\mathbb{S}^{d-1}$ defined by $\Psi(\psi)=\psi(e_{1})$.

\begin{lem}
\label{lem:HaarMeasures}
The image measure $\mu\Psi^{-1}$ is equal to $\sigma$.
\end{lem}

Assume $U$ is distributed according to $\sigma$. By Lemma \ref{lem:HaarMeasures},
$$
\E \| B_{U}g \|_{2}^{2} = \int_{\mathcal{O}(d)} \| B_{\Psi(\psi)}g  \|_{2}^{2} d\mu(\psi) ~.
$$
For any element $\psi$ of $\mathcal{O}(d)$, set $v_{1}=\psi(e_{1})$. Hence, we replace $\|B_{\Psi(\psi)}g \|_{2}^{2}$ with (\ref{ProjectionBases}):
$$
\E \| B_{U}g \|_{2}^{2} = \sum_{i=1}^{\ell(k)} \int_{\mathcal{O}(d)} \left( \int_{\mathbb{S}^{d-1}} g\circ\psi(x) S_{i}(x) d\sigma(x) \right)^{2} d\mu(\psi) ~.
$$
It suffices now to apply Lemma 2.2 of \cite{K} ensuring that each term of the above sum is equal to the ratio $\|g\|_{2}^{2}$ divided by the dimension of $\mathcal{S}_{k}$. So,
$$
\E \| B_{U}g \|_{2}^{2} = \frac{\ell(k)}{\dim\mathcal{S}_{k}} \|g\|_{2}^{2} ~.
$$
We achieve the proof of Proposition \ref{prop:contraction} with the following identities. The first one is well-known while the second one can be easily deduced from the proof of Lemma 3.1 of \cite{K}:
$$
\dim\mathcal{S}_{k} = \frac{d-2+2k}{d-2+k} { d+k-2 \choose d-2 } \; \mbox{ and } \; \ell(k) = \dim\mathcal{S}_{k}^{e_{1}} = { d+k-2 \choose d-2 } ~.
$$

This section ends with the proofs of Lemmas \ref{lem:Projection}, \ref{lem:Bases} and \ref{lem:HaarMeasures}.

\begin{proof}[Proof of Lemma \ref{lem:Projection}.]
Let $v\in\mathbb{R}^{d}$ and $g\in\mathcal{S}_{k}$. Using $\sigma\pi_{v}^{-1}=\sigma$ and $f\in\mathcal{S}_{k}^{v}$, we can write
\begin{eqnarray*}
\int_{\mathbb{S}^{d-1}} g(\pi_{v}x) f(x) d\sigma(x) & = & \int_{\mathbb{S}^{d-1}} g(\pi_{v}x) f(\pi_{v}x) d\sigma(x) \\
& = & \int_{\mathbb{S}^{d-1}} g(x) f(x) d\sigma(x)
\end{eqnarray*}
from which $\langle g-B_{v}g , f\rangle=0$ follows.
\end{proof}

\begin{proof}[Proof of Lemma \ref{lem:Bases}.]
Previous notations lead to the identity $\psi\circ\pi_{e_{1}}\circ\psi^{-1}=\pi_{v_{1}}$. Thus, Lemma \ref{lem:Projection} gives the searched result:
\begin{eqnarray*}
\mathrm{Proj}_{\mathcal{S}_{k}^{e_{1}}}(g\circ\psi)\circ\psi^{-1}(x) & = & \frac{1}{2}\left( g(x) + g(\psi\circ\pi_{e_{1}}\circ\psi^{-1}(x))\right) \\
& = & \frac{1}{2}\left( g(x) + g(\pi_{v_{1}}(x))\right) \\
& = & \mathrm{Proj}_{\mathcal{S}_{k}^{v_{1}}}(g) ~.
\end{eqnarray*}
\end{proof}

Lemma \ref{lem:HaarMeasures} is certainly known but we have not found it in the literature.

\begin{proof}[Proof of Lemma \ref{lem:HaarMeasures}.]
Let $U_{1}\in\mathcal{O}(d)$ and consider the endomorphism $\bar{U}_{1}$ of the orthogonal group $\mathcal{O}(d)$ defined by $\bar{U}_{1}(V)=U_{1} V$. It is then easy to see that $U_{1}\circ\Psi=\Psi\circ\bar{U}_{1}$. Since the Haar probability measure $\mu$ is invariant under $\bar{U}_{1}$, it follows the image measure $\mu\Psi^{-1}$ is invariant under $U_{1}$. This holds for any $U_{1}\in\mathcal{O}(d)$: only the Haar probability measure $\sigma$ can do it.
\end{proof}

\section{Random Steiner symmetrizations}
\label{sect:RandomSteiner}

Let $A$ be a convex body in $\mathbb{R}^{d}$ having the same volume as the unit ball $D$. The main result of this section gives a rate of convergence for
$$
S_{n}A = S_{U_{n}}(\ldots S_{U_{2}}(S_{U_{1}} A)\ldots)
$$
to $D$ when the random directions $U_{k}\in\mathbb{S}^{d-1}$, $k\geq 1$, are independent. Recall that $\sigma$ denotes the Haar probability measure on $\mathbb{S}^{d-1}$.

\begin{thm}
\label{thm:Steiner}
Assume that, for any $k\geq 1$, the distribution $\nu_{k}$ of $U_k$ is absolutely continuous with respect to $\sigma$ and its density satisfies
\begin{equation}
\label{hypodistrib2}
\frac{d\nu_{k}}{d\sigma}(u) \leq \alpha < \frac{d}{d-1}
\end{equation}
for some $\alpha>0$ and $\sigma-$a.e. $u\in\mathbb{S}^{d-1}$. Then, there exists two positive constants $c$ and $c'$ which only depend on $d$, $A$ and $\alpha$ such that, with probability $1$,
\begin{equation}
\label{rateMinkowski}
\exists n_{0}(\omega) , \; \forall n \geq n_{0} , \; d_{H}\left( S_{n}A , D \right) \leq c e^{-c'\sqrt{n}} ~.
\end{equation}
Furthermore, the first random integer $n_{0}$ from which the above inequality holds satisfies:
\begin{equation}
\label{rateSteiner}
\P ( n_{0} > m ) \leq  c e^{-c'\sqrt{m}} ~.
\end{equation}
\end{thm}

\begin{rem}
The comparison between Theorem \ref{thm:Steiner} and Klartag's result (Theorem 1.5 of \cite{K} says an implicit sequence of $n$ Steiner symmetrizations transforms $A$ into a new convex body at distance from $D$ smaller than $e^{-\delta\sqrt{n}}$) is the same as the one between Theorem \ref{thm:Minkowski} and Theorem 1.3 of \cite{K}. See the first paragraph just after Theorem \ref{thm:Minkowski} in Section \ref{sect:RateMinkowski}.
\end{rem}

\begin{rem}
The almost sure convergence (but without rate of convergence) of $S_{n}A$ to $D$ in the case $\nu_{k}=\sigma$ has been first proved by Mani-Levitska \cite{ML}. Vol\v{c}i\v{c} \cite{V} has recently extended this result to any probability measure assigning positive mass to any open subset of $\mathbb{S}^{d-1}$. Theorem \ref{thm:Steiner} improves Vol\v{c}i\v{c}'s result in two directions. First, Theorem \ref{thm:Steiner} does not require that the random directions are identically distributed. Secondly, the positivity hypothesis is relaxed here, since (\ref{hypodistrib2}) allows the $\nu_{k}$ to avoid some open subsets of $\mathbb{S}^{d-1}$. In the same way, (\ref{hypodistrib2}) completes the condition (\ref{nonstat}) of Bouchard and Fortier \cite{BF}.
\end{rem}

\begin{rem}
\label{Rmk:indephypo}
Let us note that the independence hypothesis between random directions can be slightly weakened. Indeed, Theorem \ref{thm:Steiner} still holds when the sequence $(U_{n})_{n\geq 1}$ is a time-homogeneous Markov chain on $\mathbb{S}^{d-1}$ whose transition probability kernel $\mathrm{P}$ is such that, for any $v\in\mathbb{S}^{d-1}$, the probability measure $\mathrm{P}(v,\cdot)$ satisfies the condition (\ref{hypodistrib2}). The same is true for Theorem \ref{thm:Minkowski}. See \cite{MT} for a general reference on Markov chains with continuous state space.
\end{rem}

\begin{rem}
The identity $S_{u}(rA)=r S_{u}A$, for $r>0$, allow to extend Theorem \ref{thm:Steiner} to convex bodies with any positive volume. When the volume of $A$ is null, $A$ lies in a proper subspace of $\mathbb{R}^{d}$. In this case, the Steiner symmetrization $S_{u}$ and the orthogonal projection onto $u^{\perp}$ coincide. Then, it is not difficult to prove that the rate of convergence of $S_{n}A$ to the origin is exponential.
\end{rem}

\begin{rem}
To obtain its result (Theorem 3.4 of \cite{V}), Vol\v{c}i\v{c} proved the \textit{moment of inertia} of $S_{n}A$, i.e.
$$
I(S_{n}A) = \int_{S_{n}A} \|z\|^{2} \, d\lambda_{d}(z) ~,
$$
converges to the moment of inertia of $D$ (where $\|\cdot\|$ denotes the euclidean norm). Inequality (\ref{RateSteinerNyk}) below specifies the rate of convergence;
$$
| I(S_{n}A) - I(D) | \leq R(A)^{2} d_{N}(S_{n}A , D) \leq R(A)^{2} e^{-c_{7}\sqrt{n}} \; \mbox{a.s.}
$$
\end{rem}

As it has been recalled in Section \ref{sect:Definitions}, the sequence $(L(S_{n}A))_{n}$ is nonincreasing. Hence, the sequence of corresponding expectations converges. Proposition \ref{prop:VitesseMstar} specifies its limit and its rate of convergence.

\begin{prop}
\label{prop:VitesseMstar}
There exists two positive constants $c_{1}$ and $c_{2}$ which only depend on $d$, $A$ and $\alpha$, such that for any $n$,
\begin{equation}
\label{VitesseMstar}
0 \leq \E L(S_{n}A) - 1 \leq c_{1} e^{-c_{2}\sqrt{n}} ~.
\end{equation}
\end{prop}

Let us start with proving Theorem \ref{thm:Steiner} from the above result.

\begin{proof}[Proof of Theorem \ref{thm:Steiner}.]
Recall $d_{N}$ denotes the Nikod\'ym distance. Since $S_{2n}A$ and the unit ball $D$ have the same volume, we can write:
\begin{eqnarray}
\label{DistNikS2nA}
\frac{1}{2} d_{N}(S_{2n}A , D) & = & \lambda_{d}(S_{2n}A\setminus D) \nonumber\\
& \leq & \lambda_{d}(B_{2n,n+1}(S_{n}A)\setminus D) \nonumber\\
& \leq & d_{N}(B_{2n,n+1}(S_{n}A) , D) \nonumber\\
& \leq & d_{N}(B_{2n,n+1}(S_{n}A) , L(S_{n}A)) + d_{N}(L(S_{n}A) , D) ~.
\end{eqnarray}
Now, let us bound the two terms of the sum (\ref{DistNikS2nA}). If $X_{n}$ denotes the Hausdorff distance between $B_{2n,n+1}(S_{n}A)$ and $L(S_{n}A)$, then $B_{2n,n+1}(S_{n}A)$ contains the centered ball with radius $L(S_{n}A)-X_{n}$ and is contained in the one with radius $L(S_{n}A)+X_{n}$. Hence, the first term of (\ref{DistNikS2nA}) is smaller than
\begin{equation}
\label{XnMstarSnA}
\kappa_{d} \left( (L(S_{n}A) + X_{n})^{d} - (L(S_{n}A) - X_{n})^{d} \right) ~.
\end{equation}
Inequalities $L(S_{n}A)\leq L(A)$ and $X_{n}\leq R(A)+L(A)$ allow to bound (\ref{XnMstarSnA}) by $c_{3} X_{n}$, for a suitable constant $c_{3}=c_{3}(d,A)>0$. The second term of (\ref{DistNikS2nA}) is treated in the same way:
\begin{eqnarray*}
d_{N}(L(S_{n}A) , D) & = & \lambda_{d}(L(S_{n}A) D \setminus D) \\
& = & \kappa_{d} \left( L(S_{n}A)^{d} - 1 \right) \\
& \leq & c_{4} \left( L(S_{n}A) - 1 \right) ~,
\end{eqnarray*}
for a suitable constant $c_{4}=c_{4}(d,A)>0$. Combining the previous inequalities with Proposition \ref{prop:VitesseMstar} and (\ref{EspNormeInfini2})-- from the proof of Proposition \ref{prop:VitesseMstar} --we get:
$$
\E d_{N}(S_{2n}A , D) \leq 2c_{3} a_{1} a_{2}^{n} + 2c_{4} c_{1} e^{-c_{2}\sqrt{n}}
$$
($a_{1}$ and $a_{2}$ are two positive constants depending on $d$, $A$ and $\alpha$, and $a_{2}<1$). The same upperbound holds for the expectation of $d_{N}(S_{2n+1}A , D)$ since the Steiner symmetrization is a $1-$Lipschitz function with respect to the Nikod\'ym distance (see Lemma \ref{lem:Nykodim1Lip}). To sum up, there exist $c_{5},c_{6}>0$ such that, for any $n$,
$$
\E d_{N}(S_{n}A , D) \leq c_{5} e^{-c_{6}\sqrt{n}} ~.
$$
By Markov's inequality and the Borel-Cantelli lemma, we deduce there exists $0<c_{7}<c_{6}$ such that, with probability $1$, for $n$ large enough,\begin{equation}
\label{RateSteinerNyk}
d_{N}(S_{n}A , D) \leq e^{-c_{7}\sqrt{n}} ~.
\end{equation}
Finally, the passage from the Nikod\'ym distance to the Hausdorff one is ensured by Lemma \ref{lem:Hausdorff<Nyk}. With $r=(2R(A))^{-1}$, the quantity $d_{H}(S_{n}(rA),rD)$ is smaller than $1/2$ for any integer $n$. So, Lemma \ref{lem:Hausdorff<Nyk} applies: with probability $1$,
\begin{eqnarray*}
d_{H}(S_{n}A , D) & = & r^{-1} d_{H}(S_{n}(rA) , rD) \\
& \leq & C r^{-1} d_{N}(S_{n}(rA) , rD)^{\frac{2}{d+1}} \\
& \leq & C r^{\frac{2}{d+1}-1} e^{-\frac{2c_{7}}{d+1}\sqrt{n}} ~.
\end{eqnarray*}
Statement (\ref{rateMinkowski}) follows. To get (\ref{rateSteiner}), we proceed as in the proof of Theorem \ref{thm:Minkowski}.
\end{proof}

\begin{proof}[Proof of Proposition \ref{prop:VitesseMstar}.]
Assume there exists $n$ such that $\beta:=L(S_{n}A)<1$. By Theorem \ref{thm:Minkowski}, conditionally to $S_{n}A$,
$$
B_{m,n+1}(S_{n}A) = B_{U_{m}}(\ldots B_{U_{n+1}}(S_{n} A)\ldots)
$$
converges almost surely to $\beta D$ as $m$ tends to $\infty$. Combining with the fact that the Minkowski symmetrization of a given set increases its volume (remember (\ref{MeanRadVol})), it follows
$$
vol(D) > \beta^{d} vol(D) \geq vol \left( B_{m,n+1}(S_{n}A) \right) \geq vol \left( S_{n}A \right) = vol(A) ~.
$$
This contradicts the hypothesis $vol(A)=vol(D)$ and states the lower bound of (\ref{VitesseMstar}).

The proof of the upper bound of (\ref{VitesseMstar}) requires more work. It is based on two ingredients; first, on the next lemma (Corollary 6.2 of \cite{K}) which is a particular case of a result on quermassintegrals due to Bokowski and Heil (Theorem 2 of \cite{BH}).

\begin{lem}
\label{lem:GainMstar}
Let $\varepsilon>0$ and $K\subset(1+\varepsilon)D$ be a convex body having the same volume as $D$. Then,
$$
L(K) - 1 \leq r_{d} \, \varepsilon
$$
where $r_{d}=1-\frac{1}{d^{2}}<1$.
\end{lem}

Secondly, we need to check the expectation of $\|h_{B_{n+m,n+1}(S_{n}A)}\|_{\infty}$. Since $R(S_{n}A)$ is smaller than $R(A)$, Lemma \ref{lem:InfiniNorme} applies to $S_{n}A$ instead of $A$ but with the same constant as in (\ref{InfiniNorme}), denoted by $z_{d}$. Thus, an analogue of inequality (\ref{EspNormeInfini}) is obtained: for any integers $m,n$,
\begin{equation}
\label{EspNormeInfini2}
\E \| h_{B_{n+m,n+1}(S_{n}A)} \|_{\infty} \leq a_{1} a_{2}^{m} ~,
\end{equation}
where $a_{1}=(z_{d}R(A))^{1/d}$ and $a_{2}=\left(\frac{\alpha(d-1)}{d}\right)^{1/2d}$. This latter is strictly smaller than $1$ thanks to hypothesis (\ref{hypodistrib2}).\\
Some additional constants have to be introduced. We set
$$
\gamma = \frac{r_{d}+1}{2} \, < \, 1 \, , \; b = \frac{\log \gamma}{\log a_{2}} \, > \, 0
$$
and $m\in\mathbb{N}$ such that
$$
a_{1} a_{2}^{m} \leq \frac{1-r_{d}}{2 r_{d}} \left( M^{\star}\!(A) - 1 \right) ~.
$$
Thus, by induction, we define a sequence of integers $(m_{k})_{k\geq 0}$ by
$$
m_{0} = m \; \mbox{ and } \; \forall k \in \mathbb{N} , \; m_{k+1} = \lfloor m_{k}+b \rfloor + 1
$$
(where $\lfloor x\rfloor$ denotes the integer part of $x$) and a sequence of convex bodies $(A_{k})_{k\geq 0}$ by $A_{0}=A$, $A_{1}=S_{m_{0}}A$ and for any $k\geq 1$,
$$
A_{k+1} = S_{\mathbf{m_{k}},\mathbf{m_{k-1}}+1} A_{k} \; \mbox{ where } \; \mathbf{m_{k}} = \sum_{i=0}^{k} m_{i} ~.
$$
Roughly speaking, the passage from $A_{k}$ to $A_{k+1}$ is obtained after $m_{k}$ Steiner symmetrizations. This process actually reduces the mean radius $L$. Precisely, we are going to prove that, for any $k\in\mathbb{N}$,
\begin{equation}
\label{RecMstar}
\E L(A_{k}) - 1 \leq \gamma^{k} \left( L(A) - 1 \right) ~.
\end{equation}
The case $k=0$ is obvious. Assume (\ref{RecMstar}) holds for a given $k\in\mathbb{N}$. Let us denote by $X_{k}$ the Hausdorff distance between $B_{\mathbf{m_{k}},\mathbf{m_{k-1}}+1} A_{k}$ and $L(A_{k}) D$. Thanks to (\ref{EspNormeInfini2}), the expectation of $X_{k}$ is upperbounded by $a_{1} a_{2}^{m_{k}}$. Besides, $A_{k+1}$ is included in $B_{\mathbf{m_{k}},\mathbf{m_{k-1}}+1} A_{k}$, itself included in $(X_{k}+L(A_{k})) D$. So, we can apply Lemma \ref{lem:GainMstar} to $A_{k+1}$ whose volume equals the one of $D$:
$$
L(A_{k+1}) - 1 \leq r_{d} \left( L(A_{k}) - 1 + X_{k} \right) ~.
$$
The induction hypothesis then gives
$$
\E L(A_{k+1}) - 1 \leq r_{d} \left( \gamma^{k} \left( L(A) - 1 \right) + a_{1} a_{2}^{m_{k}} \right) ~.
$$
Now, the sequence $(m_{k})_{k\geq 0}$ has been built so that
$$
a_{1} a_{2}^{m_{k}} \leq \gamma a_{1} a_{2}^{m_{k-1}} \leq \ldots \leq \gamma^{k} a_{1} a_{2}^{m} \leq \gamma^{k} \frac{1-r_{d}}{2 r_{d}} \left( L(A) - 1 \right)
$$
which finally provides
$$
\E L(A_{k+1}) - 1 \leq \gamma^{k} \left( L(A) - 1 \right) \left( r_{d} + \frac{1-r_{d}}{2 r_{d}} \right) = \gamma^{k+1} \left( L(A) - 1 \right) ~.
$$
To conclude, it suffices to extend inequality (\ref{RecMstar}) from $A_{k}$ to $S_{n}A$. So, let $n\in\mathbb{N}$ larger than $m$. Let us introduce the integer $k\geq 0$ satisfying
\begin{equation}
\label{Largestk}
\mathbf{m_{k}} \leq n < \mathbf{m_{k+1}} ~.
\end{equation}
The choice of $k$ implies on the one hand,
\begin{eqnarray*}
\E L(S_{n}A) - 1 & \leq & \E L(S_{\mathbf{m_{k}}}A) - 1 \\
& = & \E L(A_{k}) - 1 \\
& \leq & \gamma^{k} \left( L(A) - 1 \right)
\end{eqnarray*}
by (\ref{RecMstar}). On the other hand, it allows to compare $k$ and $\sqrt{n}$. Indeed,
\begin{eqnarray*}
n & < & \mathbf{m_{k+1}} \\
& \leq & (k+2) m + \frac{(k+2)(k+1)}{2}(b+1) \\
& \leq & c (k+2)^{2} ~,
\end{eqnarray*}
for a suitable constant $c>0$, only depending on $m$ and $b$. This proves the upperbound of (\ref{VitesseMstar}) for any $n\geq m$, with $c_{1}=\gamma^{-2}(L(A)-1)$ and $c_{2}=-\frac{1}{\sqrt{c}}\log\gamma$. Finally, it suffices to increase $c_{1}$ in order to get (\ref{VitesseMstar}) for any $n$.
\end{proof}

\section{Open questions}
\label{sect:Open}

The first open question concerns the rate of convergence of the random sequence $(S_{n}A)_{n\geq 1}$ to the corresponding ball: how far from optimal the rate given by Theorem \ref{thm:Steiner} is ? However no (strict) contraction property for Steiner has been exhibited, one may expect an exponential rate.

Corollary 2 and Lemma 3.4 of \cite{BF} suggest that the a.s convergence of $(S_{n}A)_{n\geq 1}$ takes place, for i.i.d. directions, provided the support of the common distribution contains a nonempty open set of the sphere $\mathbb{S}^{d-1}$. Is this condition sufficient to receive an assessment of speed of convergence ?

What about the rate of convergence of $(S_{n}A)_{n\geq 1}$ and $(B_{n}A)_{n\geq 1}$ when $A$ is only assumed to be a compact set, or a set of finite measure ?

Is there exist a stronger theorem of equivalence ensuring some relation between the rates of convergence of both sequences $(S_{n}A)_{n\geq 1}$ and $(B_{n}A)_{n\geq 1}$ ?

The counter-example exhibited in \cite{BBGV} proves that an asimptotically  uniformly distributed non random sequence $(u_{n})_{n\geq 1}$ on $\mathbb{S}^{1}$ does not always round off any given convex body. It would be interesting to find a reasonable strengthening of this condition which implies an asymptotic rounding of any convex bodies.

\section{Acknowledgments}

The authors wish to thank all participants of working seminar on Stochastic Geometry of the university Lille 1 for their support and useful discussions.

\appendix
\section{Appendix: metrics on $\mathcal{K}_{d}$}
\label{appendix}

The Hausdorff distance provides a bridge between convex bodies and their support functions. Precisely, the mapping $\phi: A\mapsto f_A$ is an isometry from $(\mathcal{K}_d,\,d_H)$ onto the subset $\phi(\mathcal{K}_d)$ of the space of continuous functions on $\mathbb{S}^{d-1}$ endowed with the $L^{\infty}$ norm. See \cite{G} p.84.

\begin{lem}
\label{lem:HausdorffSupportFct}
Let $A,B\in\mathcal{K}_d$. Then,
$$
d_{H}(A , B) = \| f_A - f_B \|_{\infty} ~.
$$
\end{lem}

The reason why the Nikod\'ym distance is used in this paper lies in the fact that Steiner symmetrization is $1-$Lipschitz with respect to it. See Lemma 2.2 of \cite{V}.

\begin{lem}
\label{lem:Nykodim1Lip}
Let $A,B\in\mathcal{K}_d$ and $u\in\mathbb{S}^{d-1}$. Then,
$$
d_N (S_uA, S_uB) \leq d_N (A, B) ~.
$$
\end{lem}

This appendix ends with two inequalities comparing the Hausdorff and Nikod\'ym distances.

\begin{lem}
\label{lem:Nyk<Hausdorff}
There exists a positive constant $C=C(d,R)$ such that for all $A,B\in\mathcal{K}_d$ included in the ball $B(0,R)$:
$$
d_N(A,B) \leq C d_H(A,B).
$$
\end{lem}

\begin{proof}
Let us denote by $\Delta_{1},\ldots,\Delta_{d}$ the one-dimensional segments, centered and with length $2\varepsilon$ along the vectors of the canonical basis of $\mathbb{R}^{d}$:
$$
[-\varepsilon,\varepsilon]^{d} = \Delta_{1}\oplus\ldots\oplus\Delta_{d} ~.
$$
Let us set $B_{i}=B\oplus\Delta_{1}\oplus\ldots\oplus\Delta_{i}$ for $1\leq i\leq d$ and $B_{0}=B$. Whenever $\varepsilon>d_H(A,B)$, we can write:
\begin{equation}
\label{Bi+1}
\lambda^{d}(A\setminus B) \leq \lambda^{d}(B_{d}\setminus B) \leq \sum_{i=0}^{d-1} \lambda^{d}(B_{i+1}\setminus B_{i}) ~.
\end{equation}
Now, $B_{i}$ is a convex set included in the centered (euclidean) ball with radius $R+\sqrt{d}\varepsilon$. So, its $(d-1)-$dimensional volume is smaller than $d \kappa_{d} (R+\sqrt{d}\varepsilon)^{d-1}$, where $\kappa_{d}$ denotes the volume of the $d-$dimensional unit ball. Hence,
$$
\lambda^{d}(B_{i+1}\setminus B_{i}) \leq 2\varepsilon d \kappa_{d} (R+\sqrt{d}\varepsilon)^{d-1} ~.
$$
Since $d_H(A,B)<2R$, we can take $\varepsilon<2R$. Then, $\lambda^{d}(B_{i+1}\setminus B_{i})$ is upperbounded by $C \varepsilon$ where $C=C(d,R)>0$. By (\ref{Bi+1}),
$$
\lambda^{d}(A\setminus B) \leq d C \varepsilon ~.
$$
The same inequality holds for $\lambda^{d}(B\setminus A)$. The searched result follows when $\varepsilon\searrow d_H(A,B)$.
\end{proof}

\begin{lem}
\label{lem:Hausdorff<Nyk}
Let $A$ be a convex body having the same volume as $D$ and such that $d_{H}(A,D)\leq\frac{1}{2}$. Then there exists a positive constant $C=C(d)$ such that
$$
d_H(A,D) \leq C d_N(A,D)^{\frac{2}{d+1}} ~. 
$$ 
\end{lem}

\begin{proof}
Let $r=d_H(A,D)$. There exists a vector $a\in A$ such that $\|a\|_{2}=1\pm r$. We only treat the case $\|a\|_{2}=1+r$ since the case $\|a\|_{2}=1-r$ is similar. Let us consider the semi-infinite cone $K_{1}$ formed by all rays emanating from $a$ and intersecting the ball $(1-r)D$, and the outer half-space $K_{2}$ which is tangent to $D$ at $a/\|a\|_{2}$. An elementary calculation shows the set $K_1\cap K_2$ is a right cone with height $r$ over a $(d-1)-$dimensional ball with radius larger than $\sqrt{r}/4$ (because $r\leq 1/2$). Hence, the volume of $A\setminus D$ which contains $K_1\cap K_2$, is larger than $Cr^{\frac{d+1}{2}}$ where $C=C(d)$ is a positive constant. To conclude, we use the identity
$$
d_{N}(A,D) = 2 \lambda^{d}(A\setminus D)
$$
since $A$ and $D$ have the same volume.
\end{proof}

{\small \bibliographystyle{plain}
\bibliography{SteinerBibli}}

\end{document}